\documentclass[12pt]{amsart}
\usepackage[utf8]{inputenc}
\usepackage{amsmath}
\usepackage{amssymb}
\usepackage{amsthm}
\usepackage[
backend=biber,
style=alphabetic,
]{biblatex}
\usepackage{chngcntr}
\usepackage{enumitem}
\usepackage{extpfeil}
\usepackage[margin=2cm]{geometry}
\usepackage{hyperref}
\usepackage{parskip}
\usepackage{mathrsfs}
\usepackage{mathtools}
\usepackage{tikz}
\usepackage{tikz-cd}

\hypersetup{
  colorlinks=true,
  linkcolor=blue,
  citecolor=blue
}

\newcommand{\Z}{\mathbb Z}
\newcommand{\C}{\mathbb C}
\newcommand{\F}{\mathbb F}
\newcommand{\PP}{\mathbb P}
\newcommand{\E}{\mathbb E}

\theoremstyle{plain}
\newtheorem{theorem}{Theorem}
\newtheorem{lemma}[theorem]{Lemma}
\newtheorem{corollary}[theorem]{Corollary}
\newtheorem{proposition}[theorem]{Proposition}

\theoremstyle{definition}
\newtheorem{definition}[theorem]{Definition}
\newtheorem{example}[theorem]{Example}
\newtheorem{remark}[theorem]{Remark}

\numberwithin{theorem}{section}

\newcommand{\Hom}{\operatorname{Hom}}
\newcommand{\Sur}{\operatorname{Sur}}
\newcommand{\Aut}{\operatorname{Aut}}
\newcommand{\coker}{\operatorname{coker}}

\newcommand{\details}[1]{}
\newcommand{\intuition}[1]{}

\title[Cokernels of Nearly-symmetric matrices]{Nearly-Symmetric Matrices in the Cohen-Lenstra Universality Class}

\author{Elia Gorokhovsky}
\address{Harvard University, Cambridge, USA}

\date{\today}

\begin{document}

\begin{abstract}
In this paper, we study cokernels of random $n\times n$ matrices over $\Z$ with symmetry conditions determined by fixed alternating bilinear forms on $\Z^n$. These include perturbations of random symmetric matrices at a very small (but unbounded with $n$) number of entries. We show that, subject to fairly weak conditions on the distributions of the entries, the distribution of these cokernels converges weakly to the Cohen-Lenstra distribution, which is the limiting distribution of cokernels of random matrices with no symmetry constraints. This result demonstrates that the cokernel distributions of symmetric matrices are quite sensitive to small perturbations of the symmetry conditions.
\end{abstract}

\maketitle

\section{Introduction}

We study the limiting distribution of cokernels of large random matrices with integer entries. The motivation for this problem comes from arithmetic statistics. Cohen and Lenstra \cite{cohenlenstra1984} conjectured, based on numerical evidence, a distribution for the $p$-Sylow subgroups of class groups of imaginary quadratic number fields for odd primes $p$: for a finite abelian $p$-group $G$, the \textit{Cohen-Lenstra random group} $\Gamma_{CL}^{(p)}$ satisfies \[
\PP[\Gamma_{CL}^{(p)} \cong G] \propto \frac{1}{|\Aut(G)|}.
\]
Later, Friedman and Washington \cite{FriedmanWashington+1989+227+239} observed that this distribution, now called the Cohen-Lenstra distribution, arises as the limiting distribution of cokernels of matrices of the form $I - M$, where $M$ is Haar-random in $\operatorname{GL}_{2g}(\Z_p)$, the group of invertible $2g \times 2g$ matrices over the $p$-adic integers, as $g \to \infty$.

Wood \cite{wood2019matrices} and Nguyen and Wood \cite{nguyenRandomIntegralMatrices2022} found that the Cohen-Lenstra distribution is \textit{universal}, in the sense that one obtains the same limiting distribution from ($p$-parts of) cokernels of $n \times n$ integer matrices as $n \to \infty$, as long as the entries of the matrices are independent and satisfy a weak regularity condition called \textit{$\varepsilon$-balancedness} (Definition~\ref{def:eps-bal}). Subsequent work has provided many more examples of square integer or $p$-adic random matrix ensembles, even with some dependence among the entries, whose cokernels converge to the Cohen-Lenstra distribution in the same sense, such as Laplacians of Erd\H{o}s-R\'enyi random digraphs \cite{Meszaros2020}, random band matrices with wide enough bands \cite{mészáros2024phasetransitioncokernelsrandom}, random block matrices with weak regularity conditions on the blocks \cite{gorokhovsky2024timeinhomogeneousrandomwalksfinite}, matrices with some entries fixed at zero \cite{kang2024randompadicmatricesfixed}, and certain sparse matrices \cite{meszarosCohenLenstraDistribution2025}.

At the same time, there are numerous examples of random matrix models whose cokernels do not converge to the Cohen-Lenstra distribution, such as symmetric matrices \cite{wood2017sandpile, nguyenLocalGlobalUniversality2025}, alternating matrices \cite{nguyenLocalGlobalUniversality2025}, products of independent random matrices \cite{nguyen2024products}, quadratic polynomials in Haar-random $p$-adic matrices \cite{cheong2022polynomials}, and matrices with too many entries fixed to zero \cite{kang2024randompadicmatricesfixed, mészáros2024phasetransitioncokernelsrandom}. It is natural to ask what kind of dependence among the entries of random matrices prevents their cokernels from converging to the Cohen-Lenstra distribution.

Sometimes, as in the case of symmetric matrices, the differences in limiting distributions are explained by the existence of extra structure (e.g., perfect pairings) on cokernels of certain matrices. When keeping track of these structures in addition to the group structure on the cokernel, one gets distributions more closely resembling Cohen-Lenstra \cite{hodgesDistributionSandpileGroups2024}. 

We study a new random matrix model given by matrices with the following symmetry constraints:

\begin{definition}
Let $C$ be an alternating $n\times n$ matrix over a ring $R$. We say an $n\times n$ matrix $X$ over $R$ is \textit{$C$-symmetric} if $X - X^\intercal = C$. In particular, a matrix is $0$-symmetric if and only if it is symmetric.
\end{definition}

We will give some examples of interesting random $C$-symmetric matrices later in the introduction and in more detail in Section~\ref{sect:examples}.

Our main theorem shows that the limiting distribution of ensembles of random $C$-symmetric integer matrices for ``most'' choices of $C$ is the Cohen-Lenstra distribution, in the following sense:

\begin{theorem}\label{thm:intro-result}
Let $\varepsilon > 0$. For each $n$, let $C_n$ be a deterministic $n\times n$ alternating matrix over $\Z_p$. Let $X_n$ be a random $C_n$-symmetric matrix over $\Z_p$ with independent $\varepsilon$-balanced entries on and above the diagonal. Assume that we have \[
\lim_{n\to\infty} \operatorname{rank}(C_n\mod{p}) = \infty.
\]
Then for any finite abelian $p$-group $G$, we have \[
\lim_{n\to\infty} \PP[\coker(X_n) \cong G] = \frac{1}{|\Aut(G)|} \prod_{k=1}^\infty (1 - p^{-k})
\]
In other words, the distribution of $\coker(X_n)$ converges to the Cohen-Lenstra distribution.
\end{theorem}

Theorem~\ref{thm:intro-result} is a consequence of the following stronger result (see Remark~\ref{rmk:gen-and-rank} for the relation between $\operatorname{rank}(C_n\mod{p})$ and $n - g_n$), which is equivalent to Theorem~\ref{thm:moment-convergence}:
\begin{theorem}\label{thm:intro-moments}
Let $\varepsilon > 0$. For each $n$, let $C_n$ be a deterministic $n\times n$ alternating matrix over $R = \Z$ (respectively, $R = \Z/a\Z$ for $a > 0$, or $R = \Z_p$) and let $h$ be an integer dividing all entries of $C_n$ for every $n$. Suppose the cokernel of $C_n/h$ can be generated by $g_n$ elements over $R$. Let $X_n$ be a random $C_n$-symmetric matrix over $R$ with independent $\varepsilon$-balanced entries on and above the diagonal. Assume we have \[
\lim_{n\to\infty} n - g_n = \infty
\]
For any finite abelian group $G$ (respectively, finite abelian group with exponent dividing $a$, or finite abelian $p$-group) we have \[
\lim_{n\to\infty} \E[\#\Sur(\coker(X_n), G)] = |\wedge^2 G[h]|.
\]
where $G[h]$ is the subgroup of $G$ consisting of elements of order dividing $h$.
\end{theorem}
Here $\#\Sur(\coker(X_n), G)$ denotes the number of surjective homomorphisms from $\coker(X_n)$ onto $G$. The averages of this quantity over finite abelian groups $G$ are called the \textit{moments} of the distribution of $\operatorname{coker}(X_n)$ and often provide enough information to recover the distribution uniquely. The moments computed in Theorem~\ref{thm:intro-moments} correspond to the cokernel distribution of a matrix drawn from the Haar measure on the subgroup of matrices over $\widehat{\Z}$ (the profinite completion of the integers) consisting of those matrices which are symmetric after reduction modulo $h$. This follows from the more general universality statement in Example~\ref{ex:symm-mod-h}.

The result of Theorem~\ref{thm:intro-moments} implies Theorem~\ref{thm:intro-result} through the moment method for finite abelian groups, which we review in Section~\ref{sect:background}. The method is somewhat stronger than stated in Theorem~\ref{thm:intro-moments}, in the sense that it also proves that the $p$-Sylow subgroups of $\coker(X_n)$ are asymptotically independent for any finite set of primes $p$, as well as getting versions of Theorem~\ref{thm:intro-result} over other rings like $\Z_p$ and $\Z/a\Z$. This is the same flavor of result as in, e.g., \cite{wood2017sandpile, wood2019matrices}. The distribution with moments given as in Theorem~\ref{thm:intro-moments} also has an explicit formula for the $p$-part, which is substantially more complicated due to the additional parameter $h$ and which can be found in \cite[Proposition 7.1]{sawinConjecturesDistributionsClass2024a} (with $R = \Z_p$ and $s = \varepsilon = 0$; see also \cite[Lemma 7.2]{sawinConjecturesDistributionsClass2024a}).

The $\varepsilon$-balanced condition on the entries asks that the distribution of each entry is not too concentrated modulo any prime (see Definition~\ref{def:eps-bal}). 

The condition on the ranks of the $C_n$ is true for most choices of $C_n$, in the sense that a random alternating matrix over $\F_p$ tends to have large rank as the size of the matrix grows. So, if each $C_n$ is chosen at random, the desired condition on the sequence of ranks of $C_n/h$ modulo $p$ is satisfied with probability 1. We explain this in more detail in Example~\ref{ex:symm-mod-h}. 

We do not determine the limiting cokernel distribution when $C_n$ is close to the zero matrix, but we do observe in Theorem~\ref{thm:not-symmetric} that, in the setting of Theorem~\ref{thm:intro-result}, the limiting distribution of $\coker(X_n)$ agrees with the limiting distribution of cokernels of symmetric matrices if and only if $C_n$ converges to 0 in the $p$-adic sense. This result demonstrates that the distribution of cokernels of symmetric matrices is highly sensitive to small perturbations in the symmetry of the entries. We also observe in Theorem~\ref{thm:not-cl} that if the rank of $C_n/h$ is bounded modulo any prime in the setting of Theorem~\ref{thm:intro-moments}, then one does not obtain the $|\wedge^2 G[h]|$ limiting moments as in the conclusion of that theorem. This demonstrates that the assumption $\lim_{n\to\infty} \operatorname{rank}(C_n \mod{p}) = \infty$ in Theorem~\ref{thm:intro-result} is necessary, and some assumption like $\lim_{n\to\infty} (n - g_n) = \infty$ in Theorem~\ref{thm:intro-moments} is necessary as well.

Over $R = \F_2$, when $C$ is the matrix with ones in every off-diagonal entry, $C$-symmetric matrices appear ``in nature'' as tournament matrices, i.e., adjacency matrices of tournament graphs.

Another natural setting for $C$-symmetric matrices over $\F_2$ comes again from number theory. R\'edei \cite{Rédei1934} observed that the 4-torsion in the class group of a quadratic number field can be studied by means of a matrix whose entries are quadratic residue symbols evaluated at pairs of ramified primes. By quadratic reciprocity, these \textit{R\'edei matrices} are $C$-symmetric, with $C$ having nonzero entries when the corresponding primes are both congruent to 3 modulo 4. Later, Gerth \cite{gerth4classRanksQuadratic1984} used this interpretation to study the distribution of this 4-torsion in families of quadratic number fields with a fixed number of ramified primes. Gerth's argument studied how the rank of a R\'edei matrix changes when a new random row and column are added. See \cite{Koymans2020EffectiveCO} for a proof using this strategy with effective error bounds. 

The author and Liu \cite{gorokhovskyliu4rank} gave an alternate effective computation of the limiting corank distribution for $C$-symmetric matrices over finite fields which is more similar in spirit to the proof here.

This paper is organized as follows. After discussing notation and terminology in Section~\ref{sect:notation}, in Section~\ref{sect:background} we give background on the random abelian groups that appear in this paper, as well as a brief exposition of the general strategy we will use to determine limiting cokernel distributions of our random matrix ensembles. This strategy requires computing the probability that $F\circ X_n = 0$ for every finite abelian group $G$ and map $F\colon \Z^n \to G$. In Section~\ref{sect:reduction}, we introduce a condition on the map $F$ called \textit{isotropy} which allows us to reduce this computation to a result of Wood \cite{wood2017sandpile}. In Section~\ref{sect:Fp}, we describe the isotropy condition in the setting where our matrices have entries in a finite field and use this to study $C$-symmetric matrices when $C$ is close to the zero matrix. In Section~\ref{sect:iso-prob}, we compute the probability that a random map $F\colon \Z^n \to G$ is isotropic and thereby prove our main result. Finally, in Section~\ref{sect:examples}, we give three examples of ensembles of $C$-symmetric matrices to showcase the power of our result.   

\section{Notation and Terminology}\label{sect:notation}

If $G$ is an abelian group and $p$ is a prime, we write $G_p$ for the $p$-Sylow subgroup of $G$. If $h$ is a positive integer, we write $G[h]$ for the set of elements of $G$ of order dividing $h$.

If $G, H$ are groups, we write $\Hom(G, H)$ for the set of homomorphisms from $G$ to $H$, we write $\Sur(G, H)$ for the set of surjective homomorphisms from $G$ to $H$, and we write $\Aut(G)$ for the set of automorphisms of $G$.

All rings are assumed to be commutative, and all nonzero rings are assumed to have 1.

We use $\PP$ to denote probability of an event and $\E$ to denote expectation of a random variable.

When $R$ is a ring and $G$ is an $R$-module, we use $G^* = \Hom(G, R)$ to denote the dual $R$-module. We view the free $R$-module $R^n$ as having the natural basis, so we have a natural isomorphism $R^n \cong (R^n)^* \cong (R^*)^n$ for any $n$. If $F\colon G \to H$ is a map of $R$-modules, we write $F^*\colon H^* \to G^*$ as the dual map.

We view an $n\times n$ matrix $M$ over $R$ as a map $M\colon (R^n)^* \to R^n$. We write $M^\intercal$ for the transpose map $M^*\colon (R^n)^* \to (R^n)^{**} = R^n$, and we write $\operatorname{coker}(M) = R^n/M((R^n)^*)$ for the cokernel.

We denote by $[n]$ the set $\{1, \dots, n\}$.

\section{Background on Cokernel Distributions and the Moment Method}\label{sect:background}

In this section, we give an exposition of the general strategy developed by Wood (which has since become standard) for determining the limiting distribution of cokernels of various random matrix models at finitely many primes. We will also record the results from the literature that we will need for our computations.

The primary tool that makes cokernel distributions tractable is the \textit{moment method}. Suppose that $\Gamma$ is a random group, and $G$ is a fixed group. The quantity $\E[\#\Sur(\Gamma, G)]$, the expected number of surjections from $\Gamma$ to $G$, is called the \textit{$G$-moment} of $\Gamma$.

We say an abelian group $G$ is \textit{finite mod $a$} for a positive integer $a$ if $G \otimes \Z/a\Z$ is a finite abelian group. 

\begin{theorem}[Consequence of {\cite[Theorem 3.1]{wood2019matrices}}, see also \cite{sawin2024momentproblemrandomobjects}]\label{thm:moment-robustness}
Let $a > 0$ be an integer. Let $\Gamma_1, \Gamma_2, \dots$ be a sequence of random abelian groups which are finite mod $a$ and $\Gamma_\infty$ be a random abelian group which is finite mod $a$. Let $A$ be the set of isomorphism classes of finite abelian groups with exponent dividing $a$. If for every $G \in A$ we have \[
\lim_{n\to\infty} \E[\#\Sur(\Gamma_n, G)] = \E[\#\Sur(\Gamma_\infty, G)] \leq |\wedge^2 G| \tag{$*$}
\]
then for every $H \in A$ we have \[
\lim_{n\to\infty} \PP[\Gamma_n \otimes \Z/a\Z \cong H] = \PP[\Gamma_\infty \otimes \Z/a\Z \cong H] \tag{$**$}
\]
\end{theorem}

In \cite[Theorem 3.1]{wood2019matrices}, Wood also proves the existence of a random group with given well-behaved moments. In \cite{sawin2024momentproblemrandomobjects}, Sawin and Wood prove a much more general version of this theorem for random objects in a diamond category and give a formula for the distribution of a random object with given moments. In this paper, the moments we will find already match with moments from a known random group, so we will not need the existence statement.

\begin{remark}\label{rmk:weak-convergence}
We say an abelian group $G$ is \textit{well-behaved} if it is finite mod $a$ for every positive integer $a$. Well-behaved abelian groups include all finitely generated abelian groups. One can put a topology on the set of isomorphism classes of well-behaved abelian groups by specifying a basis of open sets of the form \[
U_{a, H} = \{G \text{ well-behaved abelian} \mid G \otimes \Z/a\Z \cong H\}\]
for positive integers $a$ and finite abelian groups $H$ of exponent dividing $a$. A version of this topology on sets of isomorphism classes of profinite groups was studied in \cite{liu2018freegroup}. The condition $(**)$ for all choices of $a$ and $H$ is equivalent to the assertion that the (distributions of) $\Gamma_n$ converge weakly to (the distribution of) $\Gamma_\infty$ with respect to this topology. For this reason, we will sometimes say that a sequence of random groups weakly converges to a distribution to convey that $(**)$ is satisfied for all $a, H$. Theorem~\ref{thm:moment-robustness} says that to prove weak convergence, it suffices to check $(*)$ for all choices of $a$ and $G$ of exponent dividing $a$. See \cite[Section 3.5]{gorokhovsky2024timeinhomogeneousrandomwalksfinite} for more details.
\end{remark}

We now define two random groups that will be of relevance to this paper. 

Let $u \geq 0$ be an integer. For a prime $p$, define a random group $\Gamma_{CL}^{(p, u)}$ by \[
\PP[\Gamma_{CL}^{(p, u)} \cong G] = \frac{1}{|G|^u|\Aut(G)|}\prod_{k=1}^\infty(1 - p^{-k-u})
\]
for any finite abelian $p$-group $G$, and define \[
\Gamma_{CL}^{(u)} \coloneqq \prod_{p\text{ prime}} \Gamma_{CL}^{(p, u)}
\]
\details{Note that $\Gamma_{CL}^{(u)}$ is well-behaved because tensoring with $\Z/a\Z$ kills all terms in the product not indexed by primes dividing $a$.} Cohen and Lenstra \cite{cohenlenstra1984} conjectured that the random group $\Gamma_{CL}^{(p, 0)}$ describes the distribution of $p$-parts of class groups of imaginary quadratic number fields, and the random group $\Gamma_{CL}^{(p, 1)}$ describes the distribution of class groups of real quadratic number fields. Taking the product over all primes $p$ to define $\Gamma_{CL}^{(u)}$ corresponds to the assumption that the $p$-parts of these class groups behave ``independently'' for different primes $p$.

For a prime $p$, define a random group $\Gamma_{\text{sandpile}}^{(p)}$ by \begin{align*}
\PP[&\Gamma_{\text{sandpile}}^{(p)} \cong G] \\
&= \frac{\#\{\text{symmetric, bilinear, perfect }\phi\colon G \times G \to \C^*\}}{|G||\Aut(G)|}\prod_{k=1}^\infty (1 - p^{1 - 2k})
\end{align*}
for finite abelian $p$-groups $G$, and define \[
\Gamma_{\text{sandpile}} \coloneqq \prod_{p\text{ prime}} \Gamma_{\text{sandpile}}^p
\]
Wood \cite{wood2017sandpile} showed that $\Gamma_{\text{sandpile}}^{(p)}$ is the limiting distribution of the $p$-parts of sandpile groups of quite general Erd\H{o}s-R\'enyi random graphs (i.e., the cokernels of the reduced Laplacians of such graphs). The term counting perfect pairings on $G$ is related to the fact that the torsion part of the cokernel of a symmetric matrix carries a natural perfect pairing. Hodges \cite{hodgesDistributionSandpileGroups2024} showed that Wood's computation can be refined: if one instead keeps track of both the sandpile group and its duality pairing, one gets a Cohen-Lenstra-like distribution on isomorphism classes of groups-with-pairing.

\begin{lemma}[{\cite[Lemma 3.2]{wood2019matrices}, \cite[Theorem 8.3, Corollary 9.2]{wood2017sandpile}}, see also {\cite[Theorem 2]{clancyCohenLenstraHeuristic2015}}]\label{lem:known-moments}
For a finite abelian group $G$, we have \[
\E[\#\Sur(\Gamma_{CL}^{(u)}, G)] = |G|^{-u}
\]
and \[
\E[\#\Sur(\Gamma_{\text{sandpile}}, G)] = |\wedge^2 G|
\]
\end{lemma}

Wood and others have shown that these moments arise from matrix models with fairly weak conditions on the entries, which we now define.

\begin{definition}\label{def:eps-bal}
Let $N$ be a random variable valued in a ring $R$ and let $\varepsilon > 0$. We say $N$ is \textit{$\varepsilon$-balanced} if for any maximal ideal $\mathfrak p$ of $R$, and for any residue class $r \in R/\mathfrak p$, we have \[
\PP[N \equiv r \pmod{\mathfrak p}] \leq 1 - \varepsilon
\]
\end{definition}

We will say a random matrix with entries in $R$ is $\varepsilon$-balanced if its entries are independent and $\varepsilon$-balanced. We will say a $\varepsilon$-balanced ($C$-)symmetric matrix with entries in $R$ is a random matrix whose entries above and on the diagonal are independent and $\varepsilon$-balanced. Note that, in a regrettable abuse of notation, an $\varepsilon$-balanced ($C$-)symmetric matrix is never $\varepsilon$-balanced. In this paper, we restrict ourselves to the cases $R = \Z$, $R = \Z_p$ for $p$ prime, and $R = \Z/a\Z$ for $a$ a positive integer. 

\begin{theorem}[{\cite[Theorem 2.9]{wood2019matrices}}]\label{thm:iid-moments}
Let $\varepsilon > 0$ and $G$ be a finite abelian group with an $R$-module structure for $R = \Z$, $R = \Z_p$, or $R = \Z/a\Z$. For each natural number $n$, let $X_n$ be an $\varepsilon$-balanced $n\times(n+u)$ matrix with entries in $R$. There are $c, K > 0$ depending only on $\varepsilon$ and $G$ such that \[
\left|\E[\#\Sur(\coker(X_n), G)] - |G|^{-u}\right| \leq Ke^{-cn}
\]
In particular, $\coker(X_n)$ converges weakly to $\Gamma_{CL}^{(u)}$.
\end{theorem}

\begin{theorem}[Consequence of {\cite[Theorem 6.1]{wood2017sandpile}}]\label{thm:symmetric-moments}
Let $\varepsilon > 0$ and $G$ be a finite abelian group with an $R$-module structure for $R = \Z$, $R = \Z_p$, or $R = \Z/a\Z$. For each natural number $n$, let $X_n$ be an $\varepsilon$-balanced \textit{symmetric} $n\times n$ matrix with entries in $R$. There are $c, K > 0$ depending only on $\varepsilon$ and $G$ such that \[
\left|\E[\#\Sur(\coker(X_n), G)] - |\wedge^2G|\right| \leq Ke^{-cn}
\]
In particular, $\coker(X_n)$ converges weakly to $\Gamma_{\text{sandpile}}$.  
\end{theorem}

\begin{remark}
If $R = \Z$ or $\Z_p$, surjections $\coker(X_n) \twoheadrightarrow G$ always factor uniquely through $\Z/a\Z$ if the exponent of $G$ divides $a$. So, without loss of generality we will restrict ourselves to $R = \Z/a\Z$ for the rest of the paper.
\end{remark}

In fact, Nguyen and Wood \cite{nguyenRandomIntegralMatrices2022, nguyenLocalGlobalUniversality2025} show that in both cases, one actually has strong (setwise) convergence of the cokernels to the limiting distribution, although they impose an i.i.d. assumption on the entries. While the results of this paper are limited to weak convergence, it would be interesting to see if strong convergence still holds.

We give a brief summary of the ideas in the proof of Theorem~\ref{thm:symmetric-moments} that will be useful to us. A similar (but simpler) approach can also be used to prove Theorem~\ref{thm:iid-moments}. The first step is to notice that surjections $\coker(X_n) \twoheadrightarrow G$ correspond to surjections $R^n \twoheadrightarrow G$ which vanish on the image of $X_n$. Thus, we have the well-known result:

\begin{lemma}\label{lem:moment-sum}
Let $X_n$ be a random $n\times n$ matrix over $R$ and $G$ a finite abelian group with $R$-module structure. Then \[
\E[\#\Sur(\coker(X_n), G)] = \sum_{F\colon R^n \twoheadrightarrow G} \PP[F\circ X_n = 0]
\]
\end{lemma}

Ideally, one could compute $\PP[F\circ X_n = 0]$ directly for each $F$, then compute the sum. If $F$ is sufficiently general, one hopes that the distribution of $F\circ X_n$ is easy to understand, since composition with $F$ mixes the many independent sources of randomness comprising $X_n$. We make precise what we need from the notion of ``sufficiently general'':

\begin{definition}
For $\sigma \subset [n]$, write $R^n_{\sigma}$ for the submodule of $R^n$ spanned by the standard basis elements \textit{not} indexed in $\sigma$. 

We say $F\in \Hom(R^n, G)$ is a \textit{code of distance $w$} if for every $\sigma \subset [n]$ with $\#\sigma < w$, the map $F|_{R^n_{\sigma}}\colon R^n_{\sigma} \to G$ is surjective. 
\end{definition}

Then we have:

\begin{lemma}[{\cite[Lemma 4.1]{wood2017sandpile}}]\label{lem:symm-codes}
Let $\varepsilon, w > 0$, let $a$ be a positive integer and $R = \Z/a\Z$, and let $G$ be a finite abelian group of exponent dividing $a$. There is a constant $K > 0$ such that the following holds.

Let $X_n$ be a random $\varepsilon$-balanced symmetric $n\times n$ matrix, viewed as a map $(R^n)^* \to R^n$. Let $F$ be a code of distance $w$. Let $A \in \Hom((R^n)^*, G)$. For all $n$ we have \[
\left|\PP[F\circ X_n = 0] - |\wedge^2G||G|^{-n}\right| \leq \frac{K\exp\left(\frac{\varepsilon w}{4a^2}\right)}{|G|^n}
\]
and \[
\PP[F\circ X_n = A] \leq K|G|^{-n}
\]
\end{lemma}

The second part of Lemma~\ref{lem:symm-codes} is not needed for the main term in the moment calculation from Lemma~\ref{lem:moment-sum} coming from codes $F$. However, it will come up when we try to handle maps $F\colon R^n \to G$ that are not codes. The idea is that any map $F\colon R^n \to G$ is a code onto some subgroup $H \leq G$ after restricting appropriately. If the index of $H$ in $G$ is small enough, some control over the behavior of the restricted map into $H$ is enough to show that the contribution of non-codes is negligible.

To do this more carefully, Wood stratifies non-codes $F\colon R^n \to G$ by the index of the largest subgroup $H \leq G$ onto which they are a code:

\begin{definition}
For an integer $D$ with prime factorization $D = \prod_i p_i^{e_i}$, we define $\ell(D) = \sum_i e_i$.

The \textit{$w$-depth} of a map $F\colon R^n \to G$ is the maximal positive integer $D$ such that there is a subset $\sigma \subset [n]$ with $\#\sigma < \ell(D)w$ and $D = [G : F(R^n_\sigma)]$, or 1 if no such $D$ exists.
\end{definition}

We observe that a map of $w$-depth 1 is a code of distance $w$. We have a bound on the number of maps of each depth more than 1:

\begin{lemma}[{\cite[Lemma 5.2]{wood2017sandpile}}]\label{lem:num-depth}
There is a constant $K$ depending on $G$ such that if $D > 1$, there are at most \[
K\binom{n}{\lceil \ell(D)w\rceil - 1}|G|^nD^{-n + \ell(D)w}
\]
maps $R^n \to G$ of $w$-depth $D$.
\end{lemma}

Then Lemma~\ref{lem:symm-codes} can be used to bound $\PP[F\circ X_n = 0]$ when $F$ is a code. We say a property $\mathcal{P}$ of square random matrices over $R$ is \textit{stable} if for any subset $\eta \subset [n]$ and any random matrix $X$ with $\mathcal{P}$, the random matrix $X_{\eta\eta} = (X_{ij})_{i, j \in \eta}$ has $\mathcal{P}$. Examples of stable properties include: \begin{itemize}
    \item the vacuous property satisfied by all matrices;
    \item the property that $X$ is symmetric with probability 1;
    \item the property that there exists a matrix $C$ such that $X$ is $C$-symmetric with probability 1.
\end{itemize}

We also say $X$ \textit{has many independent entries} if the pairs $(X_{ij}, X_{ji}) \in R^2$ are independent as $\{i, j\}$ ranges over two-element subsets of $[n]$. In particular, $X_{ij}$ is independent of $X_{k\ell}$ unless $\{i, j\} = \{k, \ell\}$.

The following lemma is proved for symmetric matrices in \cite[Lemma 5.4]{wood2017sandpile}, but it follows by examining the proof that it holds in the greater generality given here. 

\begin{lemma}\label{lem:bound-prob-depth}
Let $\varepsilon > 0$, let $a$ be a positive integer and $R = \Z/a\Z$, and let $G$ be a finite abelian group of exponent dividing $a$. 

Let $X_n$ be a random $n\times n$ matrix with a stable property $\mathcal{P}$ with many independent $\varepsilon$-balanced entries.

Let $F\colon R^n \to G$ have $w$-depth $D > 1$ and $[G : F(R^n)] < D$.

Suppose the following assumption holds: \begin{enumerate}[label=(A)]
    \item For any finite abelian group $H$ of exponent dividing $a$, any any $m\times m$ random matrix $Y_m$ with $\mathcal{P}$ with many independent $\varepsilon$-balanced entries, any code $F'\colon R^m \to H$ of distance $w$, and any map $A \in \Hom((R^m)^*, G)$, there is a constant $K'$ depending only on $a, H, w/m$ such that \[
    \PP[F'\circ Y_m = A] \leq K'|H|^{-m}
    \]
\end{enumerate}
Then there is a constant $K > 0$ depending only on $\varepsilon, a, G, w/n, D$ such that \[
\PP[F\circ X_n = 0] \leq K\left(e^{-\varepsilon}\frac{|G|}{D}\right)^{\ell(D)w - n}
\]
\end{lemma}
Finally, to prove Theorem~\ref{thm:symmetric-moments} (respectively, Theorem~\ref{thm:iid-moments}), Lemma~\ref{lem:symm-codes} (respectively, an analogous result for non-symmetric matrices), Lemma~\ref{lem:num-depth}, and Lemma~\ref{lem:bound-prob-depth} are combined to compute the right hand side of the sum in Lemma~\ref{lem:moment-sum}. We will work this out more generally in Lemma~\ref{lem:moment-from-isotropy} for the case of $C$-symmetric matrices.

\section{Reduction to Symmetric Case}\label{sect:reduction}

\begin{lemma}\label{lem:reduction}
Let $0 < \varepsilon, \delta < 1$. Let $a \in \Z_{>0}$ and $G$ a finite abelian group of exponent dividing $a$. Then there are $K, c > 0$ such that the following holds:

Let $C$ be a deterministic $n\times n$ alternating matrix with entries in $R \coloneqq \Z/a\Z$. Let $X$ be a random $C$-symmetric matrix with entries on and above the diagonal independent and $\varepsilon$-balanced. Let $F\colon R^n \to G$ be a code of distance $\delta n$. Let $A \in \Hom((R^n)^*, G)$.

If there is no $C$-symmetric matrix $M$ with $F\circ M = A$, then $\PP[F\circ X = A] = 0$.

Otherwise, we have \[
\left|\PP[F\circ X = A] - |\wedge^2G| |G|^{-n}\right| \leq \frac{K\exp(-cn)}{|G|^n}
\]
\end{lemma}
\begin{proof}
The case where there is no $C$-symmetric matrix $M$ with $F\circ M = A$ is immediate.

Suppose now that $M$ is a $C$-symmetric matrix with $F\circ M = A$. Then $X - M$ is symmetric, and since $M$ is deterministic, $X - M$ still has independent, $\varepsilon$-balanced entries above the diagonal. Moreover, $F\circ X = A$ if and only if $F \circ (X - M) = 0$. Then the result follows from \cite[Lemma 4.1]{wood2017sandpile}.
\end{proof}

We say a map $F \in \Hom(R^n, G)$ is \textit{isotropic} for an alternating $C$ if there exists a $C$-symmetric matrix $M$ with $F\circ M = 0$. The main term in the moment computation is \[
\sum_{F\text{ code}} \PP[F\circ X = 0] \approx |\wedge^2 G| \frac{\#\{F \in \Hom(R^n, G) \mid F\text{ is a code of distance }\delta n\text{ isotropic for }C\}}{|G|^n}.
\]

We will interpret the fraction on the right as the probability that a uniformly random map $R^n \to G$ is isotropic for $C$ (and is a code of distance $\delta n$). The following lemma handles all the error terms so that we can focus on computing this probability.

\begin{lemma}\label{lem:moment-from-isotropy}
Let $a \in \Z_{>0}$ and a $G$ a finite abelian group of exponent dividing $a$. Let $\varepsilon > 0$. Then there are $K, c > 0$ such that the following holds:

Let $C$ be a deterministic $n\times n$ alternating matrix with entries in $R \coloneqq \Z/a\Z$. Let $X$ be a random $C$-symmetric matrix with entries on and above the diagonal independent and $\varepsilon$-balanced. Let $\widetilde{F}$ be a uniformly random homomorphism $R^n \to G$. Then \[
\left|\E[\#\Sur(\coker(X), G)] - |\wedge^2G|\PP[\widetilde{F}\text{ is isotropic for }C]\right| \leq Ke^{-cn}
\]
\end{lemma}
\begin{proof}
Much of this calculation is identical to the proofs of, e.g., \cite[Theorem 6.1, Theorem 6.2]{wood2017sandpile} and \cite[Theorem 2.9]{wood2019matrices}. See there for more details.

Let $\delta > 0$ be small: we will choose it later to make the error decay exponentially. Split maps $F\colon R^n \to G$ into codes of distance $\delta n$ and non-codes. By Lemma~\ref{lem:moment-sum} we want to compute \begin{align*}
|\E[\#&\Sur(\coker(X), G)] - |\wedge^2G|\PP[\widetilde{F}\text{ is isotropic for }C]| \\ &= \left|\sum_{F \in \Sur(R^n, G)}\PP[F\circ X = 0] - \sum_{\substack{F \in \Hom(R^n, G) \\ F \text{ isotropic}}}\frac{|\wedge^2G|}{|G|^n}\right| \\
&\leq \left|\sum_{\substack{F \in \Sur(R^n, G) \\F \text{ code}}}\PP[F\circ X = 0] - \sum_{\substack{F \in \Sur(R^n, G) \\ F\text{ code} \\ F \text{ isotropic}}}\frac{|\wedge^2G|}{|G|^n}\right| \tag{1} \\
&\qquad+\sum_{\substack{D > 1\\D \mid |G|}}\sum_{\substack{F \in \Sur(R^n, G) \\ F\text{ depth }D}}\PP[F\circ X = 0] \tag{2}\\
&\qquad+\sum_{\substack{D > 1\\D \mid |G|}}\sum_{\substack{F \in \Sur(R^n, G) \\ F\text{ depth }D \\ F\text{ isotropic}}} \frac{|\wedge^2G|}{|G|^n} \tag{3}\\
&\qquad+\sum_{F \in \Hom(R^n, G) \setminus \Sur(R^n, G)}\frac{|\wedge^2G|}{|G|^n}\tag{4}
\end{align*}
We bound each term individually. By Lemma~\ref{lem:reduction}, we have \[
\left|\sum_{\substack{F \in \Sur(R^n, G) \\F \text{ code}}}\PP[F\circ X = 0] - \sum_{\substack{F \in \Sur(R^n, G) \\ F\text{ code} \\ F \text{ isotropic}}}\frac{|\wedge^2G|}{|G|^n}\right| \leq |G|^n \frac{K_1e^{-c_1n}}{|G|^n} = K_1e^{-c_1n} \tag{1}
\]
for some $K_1, c_1 > 0$ depending only on $a, G, \varepsilon, \delta$.

Also, by Lemma~\ref{lem:reduction}, the condition (A) of Lemma~\ref{lem:bound-prob-depth} holds, so if $F$ has $\delta n$-depth $D$, then $\PP[F\circ X = 0] \leq K_{2, D}'\exp(-\varepsilon(1 - \ell(D)\delta)n)|G|^{(\ell(D)\delta - 1)n}D^{(1 - \ell(D)\delta)n}$ for some constant $K_{2, D}' > 0$ depending only on $a, G, \varepsilon, \delta, D$. Then by Lemma~\ref{lem:num-depth} we have \[
\sum_{\substack{F \in \Sur(R^n, G) \\ F\text{ depth }D}}\PP[F\circ X = 0] \leq K_{2, D}\binom{n}{\lceil \ell(D)\delta n\rceil - 1}\exp(\varepsilon(1 - \ell(D)\delta)n)|G|^{\ell(D)\delta n}
\]
Now $\delta$ can be chosen small enough that that the right hand side is bounded by a constant multiple of $\exp(-c_2n)$ for some $c_2 > 0$ depending on $a, G, \varepsilon, D$. Since $|G|$ has finitely many divisors, we get \[
\sum_{\substack{D > 1\\D \mid |G|}}\sum_{\substack{F \in \Sur(R^n, G) \\ F\text{ depth }D}}\PP[F\circ X = 0] \leq K_2e^{-c_2n} \tag{2}
\]
for some $K_2 > 0$ depending only on $a, G, \varepsilon$. A similar, easier calculation gives \[
\sum_{\substack{D > 1\\D \mid |G|}}\sum_{\substack{F \in \Sur(R^n, G) \\ F\text{ depth }D \\ F\text{ isotropic}}} \frac{|\wedge^2G|}{|G|^n} \leq \sum_{\substack{D > 1\\D \mid |G|}}\sum_{\substack{F \in \Sur(R^n, G) \\ F\text{ depth }D}} \frac{|\wedge^2G|}{|G|^n} \leq K_3e^{-c_3n} \tag{3}
\]
for constants $K_3, c_3 > 0$ depending only on $a, G, \varepsilon$.

For $(4)$, we write  \begin{align*}
\tag{4} \sum_{F \in \Hom(R^n, G) \setminus \Sur(R^n, G)}\frac{|\wedge^2G|}{|G|^n} &\leq \sum_{H \lneq G}\sum_{F \in \Hom(R^n, H)}\frac{|\wedge^2G|}{|G|^n} \\
&= \sum_{H \lneq G} |\wedge^2G| \left(\frac{|H|}{|G|}\right)^n \\
&\leq K_42^{-n}
\end{align*}
for some constant $K_4$ depending only on $G$. Putting these bounds together gives the desired result.
\end{proof}

To facilitate the computation of the moment $|\wedge^2G|\PP[\widetilde{F}\text{ is isotropic for }C]$ in the next sections, we give a condition on $\widetilde{F}$ that is equivalent to being isotropic for $C$. In Section~\ref{sect:Fp}, we will work out explicitly what this condition says in some special cases, such as the case where $R$ is a finite field, which will justify the use of the word ``isotropic'' in the definition.

Let $C$ be a bilinear map $(R^n)^* \times (R^n)^* \to R$, where $R = \Z/a\Z$ is a finite cyclic group. Let $G = \Z/d_1\Z \times \dots \times \Z/d_n\Z$ be a finite abelian group, with $d_1 \mid d_2 \mid \dots \mid d_n \mid a$. Let $F\colon R^n \to G$ be a homomorphism. 

For $1 \leq k \leq n$, let $R_k \coloneqq R \otimes_\Z \Z/d_k\Z \cong \Z/d_k\Z$ and \[
G_k \coloneqq G \otimes_R R_k = \Z/d_1\Z \times \dots \times \Z/d_{k-1}\Z \times (\Z/d_k\Z)^{n-k+1}
\]
Let $\pi_k\colon G \to G_k$ be the projection. View $G_k$ as an $R_k$-module, so $G_k^* = \Hom(G_k, R_k)$.

Note that if $V$ is an $R$-module, then \[
\Hom(V, R) \otimes_\Z \Z/d_k\Z \cong \Hom(V, R) \otimes_R R_k \cong \Hom(V, R_k) \cong \Hom(V \otimes_R R_k, R_k),
\]
where all the isomorphisms are canonical, and the last isomorphism comes from the fact that any map $V \to R_k$ factors through $V/d_kV$.

We interpret $C$ as a map $(R^n)^* \to R^n$, which descends to a map $C_k\colon (R_k^n)^* \to R_k^n$ for each $k$ by reducing the entries modulo $d_k$. Also, $F$ descends to a map $F_k\colon R_k^n \to G_k$ for each $k$.

\begin{proposition}\label{prop:isotropy-condition}
Let $C$ be alternating. Let $F \in \Sur(R^n, G)$.

The following are equivalent:
\begin{enumerate}
    \item $F$ is isotropic for $C$.
    \item We have \[
    F_k\circ C_k \circ F_k^* = 0 \qquad \text{ for all }k = 1, \dots, n
    \]
    \item The alternating bilinear form on $G_k^*$ given by \[
    F_k\circ C_k \circ F_k^*
    \]
    vanishes on the last $n - k + 1$ coordinates of $G_k^*$ for all $k = 1, \dots, n$.
\end{enumerate}
\end{proposition}

\begin{proof}
The first condition implies the second:

After tensoring with $\Z/d_k\Z$, a $C$-symmetric matrix $M$ over $R$ becomes a $C_k$-symmetric matrix $M_k$ over $R_k$, and $F_k\circ M_k = 0$ if $F\circ M = 0$. Then $M_k - M_k^\intercal = C_k$ implies \[
F_k \circ C_k \circ F_k^* = F_k\circ M_k \circ F_k^* - F_k\circ M_k^\intercal\circ F_k^* = (F_k\circ M_k) \circ F_k^* - F_k\circ (F_k \circ M_k)^* = 0.
\]

The second condition immediately implies the third.

We prove that the third condition implies the first and second by working in coordinates. Note that none of the three conditions in the proposition depend on a choice of basis for $R^n$.

Using Smith normal form, construct a basis $e_1, \dots, e_n$ for $R^n$ so that the map $F\colon R^n \to G$ is projection $R \to \Z/d_k\Z$ in each coordinate. We will determine explicitly what conditions (1) and (3) say about the matrix of $C$ in this basis.

Let $e_{k, i}$ be the basis for $R_k^n$ obtained by reducing modulo $d_k$. Let $\phi_{k, i}$ be the dual basis for $(R_k^n)^*$ given by $\phi_{k, i}(e_{k, j}) = \delta_{ij}$. Let $g_{k, i} = f_k(e_{k, i})$ form a generating set for $G_k$ and let $\chi_{k, i}$ be a dual generating set defined by $\chi_{k, i}(g_{k, j}) = \frac{d_k}{d_i}\delta_{ij}$. Let $c_{k, ij}$ be the $ij$th entry of $C$, modulo $d_k$. 

The behavior of the map $F_k \circ C_k\circ F_k^*$ on each generator $\chi_{k, j}$ is given by
\[\begin{tikzcd}
	{G_k^*} & {(R_k^n)^*} & {R_k^n} & G_k \\
	{\chi_{k, j}} & {\frac{d_k}{\min\{d_j, d_k\}}\phi_{k, j}} & {\frac{d_k}{\min\{d_j, d_k\}}\sum_{i=1}^n c_{k, ij}e_{k, i}} & {\frac{d_k}{\min\{d_j, d_k\}}\sum_{i=1}^n c_{k, ij} g_{k, i}}
	\arrow["{F_k^*}", from=1-1, to=1-2]
	\arrow["C_k", from=1-2, to=1-3]
	\arrow["F_k", from=1-3, to=1-4]
	\arrow[maps to, from=2-1, to=2-2]
	\arrow[maps to, from=2-2, to=2-3]
	\arrow[maps to, from=2-3, to=2-4]
\end{tikzcd}\]
The condition for this map to be zero on $G_k^*$ is that \[
\frac{d_k}{\min\{d_j, d_k\}}c_{k, ij} \equiv\frac{d_k}{\min\{d_j, d_k\}}c_{ij} \equiv 0\pmod{\min\{d_i, d_k\}}
\]
for all $i, j$. Note that by antisymmetry of $C$, we may interchange $i$ and $j$ in these conditions. When $i, j \geq k$, we are asking that $c_{ij} \equiv 0\pmod{d_k}$. 

There is some redundancy in these conditions. When $i < k$ or $j < k$, the condition on $c_{k, ij}$ is already satisfied by the condition on $c_{i, ij}$. (This shows that condition (3) implies condition (2).)

Now in this basis, the condition that $F\circ M = 0$ is the same as asking that the $k$th row of $M$ is divisible by $d_k$. 

Supposing condition (3) is satisfied and $C$ is alternating, we construct a $C$-symmetric matrix $M$ by picking the entries $m_{ij}$ below and on the diagonal (i.e., with $i \geq j$) arbitrarily subject to the condition that entries $m_{ij}$ in the $i$th row are divisible by $d_i$. Then, there is a unique $C$-symmetric matrix $M$ with such a choice of upper triangular entries, given by $m_{ji} = m_{ij} - c_{ij}$ when $i \geq j$. Since $d_j \mid c_{ij}$ and $d_j \mid d_i \mid m_{ij}$, we get that $d_j \mid m_{ji}$, so we have $F\circ M = 0$.
\end{proof}
\begin{remark}\label{rmk:hom-vs-sur}
Note that the equivalence in Proposition~\ref{prop:isotropy-condition} is stated when $F$ is a surjection $R^n \to G$. We use this assumption when viewing $F$ as a coordinatewise projection. However, in Lemma~\ref{lem:moment-from-isotropy}, we need the probability that a random homomorphism $R^n \to G$ is isotropic for $C$. We will find it much more convenient to find the probability that a random homomorphism satisfies condition (3) from Proposition~\ref{prop:isotropy-condition}, whether or not this condition is equivalent to isotropy.

With some more work, the condition of Proposition~\ref{prop:isotropy-condition} can be extended to non-surjective homomorphisms. However, we do not need to do this for the result that we want because almost all (as $n\to\infty$) homomorphisms $R^n \to G$ are surjective. If $\widetilde{F}$ is a uniformly random homomorphism $R^n \to G$, we have \[
\PP[\widetilde{F} \text{ is isotropic for }C] = \PP[\widetilde{F} \text{ is isotropic for }C\text{ and surjective}] + \PP[\widetilde{F} \text{ is isotropic for }C\text{ and not surjective}]
\]
and similarly, \[
\PP[\widetilde{F} \text{ satisfies }(3)] = \PP[\widetilde{F} \text{ satisfies }(3)\text{ and is surjective}] + \PP[\widetilde{F} \text{ satisfies }(3)\text{ and is not surjective}]
\]
so that, after applying Proposition~\ref{prop:isotropy-condition}, we get \begin{align*}
    |\PP[\widetilde{F} \text{ is isotropic for }C] - \PP[\widetilde{F} \text{ satisfies }(3)]| &\leq \PP[\widetilde{F} \text{ is isotropic for }C\text{ and not surjective}] \\&\qquad+ \PP[\widetilde{F} \text{ satisfies }(3)\text{ and is not surjective}].
\end{align*}
We have \[
\PP[\widetilde{F}\text{ is not surjective}] \leq \frac{1}{|G|^n}\sum_{H \lneq G} \#\Hom(R^n, H) = \sum_{H\lneq G} \left(\frac{|H|}{|G|}\right)^n = K2^{-n}
\]
for some constant $K$ depending only on $G$. 

Thus, the term $\PP[\widetilde{F} \text{ is isotropic for }C]$ in Lemma~\ref{lem:moment-from-isotropy} can be safely replaced by $\PP[\widetilde{F} \text{ satisfies }(3)]$ without changing anything else. This change only incurs an additional exponential error term.
\end{remark}

\section{Isotropy over a Finite Field and Very Small Perturbations}\label{sect:Fp}

In this section, we specialize to the case where $C$ is divisible by $a/p$, for $p$ a prime dividing $a$. In this situation, the only nontrivial part of the isotropy condition lives ``modulo $p$''. As a consequence, we will see that when $a = p$ is prime, the isotropy condition takes a particularly nice form that justifies the use of the word ``isotropy'' to describe this condition.

\begin{lemma}\label{lem:isotropy-over-field}
Let $R = \Z/a\Z$ and let $p \mid a$ be a prime. Let $C$ be an alternating bilinear form on $R^n$ which is divisible by $b \coloneqq a/p$, and let $\overline{C}$ be the form on $(R/pR)^n \cong \F_p^n$ induced by $C/b$. Fix $r > 0$ and let $F \in \Sur(R^n, R^r)$. Let $\overline{F} \in \Sur(\F_p^n, \F_p^r)$ be the map obtained by reducing $F$ modulo $p$. Let $V \subset \F_p^n$ be the image of $\overline{F}^*\colon \F_p^r \to \F_p^n$. Then $F$ is isotropic for $C$ if and only if $V$ is isotropic for $\overline{C}$ in the sense of subspaces of $\F_p^n$.
\end{lemma}
In particular, if $a = p$, then $F$ is isotropic for $C$ if and only the image of $F^*$ is isotropic for $C$ as a subspace of $\F_p^n$. In the statement of this lemma and the rest of the section, we will freely identify $\F_p^n$ with $(\F_p^n)^*$ and similarly with $\F_p^r$.
\begin{proof}
Here $G = 1^{n-r}\times \F_p^r$, so for $k = 1, \dots, n - r$ we have $R_k = 0$. For $k = n - r + 1, \dots, n$ we have $R_k = R$ and $G_k = R^{n-k+1}$.

By Proposition~\ref{prop:isotropy-condition}, we have that $F$ is isotropic for $C$ if and only if the composition $F\circ C \circ F^*$ vanishes. We view this as a bilinear form on $(R^r)^*$. Since $C$ is divisible by $b$, the resulting bilinear form vanishes on $p(R^r)^*$, so it descends to a bilinear form on $\F_p^r$ given by $\overline{F} \circ \overline{C} \circ \overline{F}^*$. The vanishing of this bilinear form is exactly the condition that $V$ is isotropic for $\overline{C}$.
\end{proof}

If $C$ is nonzero, it is not hard to uniformly bound the probability $\PP[F \text{ is isotropic for }C]$ away from $1$ when $F$ is uniformly random. Let $c$ be the rank of $\overline{C}$.

Use notation from Lemma~\ref{lem:isotropy-over-field}.

Let $W \coloneqq \F_p^n/\operatorname{rad}(\overline{C})$ be the quotient of $\F_p^n$ by the radical of the alternating form $\overline{C}$ and let $V = \operatorname{im}(\overline{F}^*)$. Let $\pi\colon \F_p^n \twoheadrightarrow W$ be the projection. Then $\overline{C}$ descends to a nondegenerate alternating form $\pi_*\overline{C}$ on $W$. We have that $F$ is isotropic for $C$ if and only if $V$ is isotropic for $\overline{C}$, if and only if $\pi(V)$ is isotropic for $\pi_*\overline{C}$.

Let $v_1, \dots, v_r$ be the columns of $\overline{F}^*$. If $F \in \Hom(R^n, R^r)$ is uniformly random, then $v_1, \dots, v_r$ are uniformly random and independent vectors in $\F_p^n$. The images $\pi(v_1), \dots, \pi(v_r)$ are uniform in $W$. The probability that $\pi(v_1)$ and $\pi(v_2)$ pair to zero is \[
\PP[\pi(v_1) = 0] + \PP[\pi(v_1) \neq 0] \PP[\pi_*\overline{C}(\pi(v_1), \pi(v_2)) = 0 \mid \pi(v_1) \neq 0] = p^{-c} + (1 - p^{-c})p^{-1}
\]
So, \[
\PP[F \text{ is isotropic for }C] \leq p^{-1} + p^{-c} - p^{-1-c} = p^{-1} + p^{-c}(1 - p^{-1})
\]
with equality when $r = 2$. This is bounded away from 1 as long as $c \geq 1$.

\begin{remark}\label{rmk:gen-and-rank}
Suppose $C$ is an  $n\times n$ matrix over $\Z/a\Z$ whose cokernel can be generated by $g$ elements. Then for a prime $p \mid a$, the rank of $C\mod{p}$ is at least $n - g$. If $a$ is divisible by only one prime $p$, then the maximal number of generators of the cokernel of $C$ is exactly the corank of $C\mod{p}$.
\end{remark}

As a consequence of these explicit computations, we can learn something about the moments of $\coker(X)$ when $X$ is a large $\varepsilon$-balanced $C$-symmetric matrix in the intermediate regime where $C$ has very large cokernel but is not the zero matrix. In such cases, if limiting moments exist, they match neither those of $\Gamma_{sandpile}$ nor of $\Gamma_{CL}^{(0)}$. Theorem~\ref{thm:not-symmetric} shows that the $\Gamma_{sandpile}$ limiting moments can only be achieved by $C$-symmetric matrices if $C$ converges to 0 in a suitable sense, while Theorem~\ref{thm:not-cl} shows that the $\Gamma_{CL}^{(0)}$ limiting moments (more generally, the $|\wedge^2 G[h]|$ limiting moments from Theorem~\ref{thm:intro-moments}) can only be achieved by $C$-symmetric matrices if the rank of $C$ (more generally, $C/h$) mod $p$ is unbounded.

\begin{theorem}\label{thm:not-symmetric}
Let $a \in \Z_{>0}$. For each natural number $n$, let $C_n$ be a deterministic $n\times n$ alternating matrix with entries in $R \coloneqq \Z/a\Z$. Let $X_n$ be a random $C_n$-symmetric matrix with entries on and above the diagonal independent and $\varepsilon$-balanced.

Then the following are equivalent: \begin{enumerate}[label=(\arabic*)]
    \item For any finite abelian group $G$ of exponent dividing $a$, we have \[
    \lim_{n\to\infty} \E[\#\Sur(\coker(X_n), G)] = |\wedge^2 G|.
    \]
    \item We have $C_n = 0$ for all $n$ large enough.
\end{enumerate}
\end{theorem}
This theorem demonstrates that the result of Theorem~\ref{thm:symmetric-moments} (universality of cokernel for large symmetric matrices) is far from stable under small perturbations of the symmetry condition. Taking the theorem for all values of $a$ gives us that if $X_n$ is a sequence of $C_n$-symmetric matrices whose cokernels converge weakly to $\Gamma_{sandpile}$ in the sense of Remark~\ref{rmk:weak-convergence}, then $C_n$ is eventually zero modulo any integer. This does not tell us that $C_n$ is eventually zero on the nose, and we could not expect to prove such a result because weak convergence only sees ``mod-$a$'' information. If working over $\Z_p$, this condition can be interpreted as asking that $C_n$ converges to 0 in the $p$-adic sense of being eventually divisible by arbitrarily large powers of $p$.
\begin{proof}
The implication $(2) \implies (1)$ is Theorem~\ref{thm:symmetric-moments}. We will prove $\lnot(2) \implies \lnot(1)$.

Suppose that $C_n \neq 0$ for infinitely many $n$. Since $a$ has finitely many prime factors, there is a prime $p \mid a$ such that $C_n \not\equiv 0\pmod{p^f}$ for some large enough $f$ with $p^f \mid a$. Let $f$ be minimal with this property, so that $p^{f-1} \mid C_n$ for all $n$ large enough, but $p^f \nmid C_n$ for infinitely many $n$.

Without loss of generality, we may assume $a = p^f$, or else replace $R$ by $R/p^fR$. Since $p^f \nmid C_n$ for infinitely many $n$, we have that $C_n/p^{f-1}$ induces an alternating form with positive rank on $(R/pR)^n$ infinitely often. The argument above shows that for such $n$, if $F$ is a uniformly random homomorphism $R^n \to R^r$ with $r \geq 2$, then $\PP[F\text{ is isotropic for } C_n]$ is uniformly bounded away from $1$ in terms of $p$. It follows that \[
\limsup_{n\to\infty}\E[\#\Sur(\coker(X_n), (\Z/p^f\Z)^r)] < |\wedge^2(\Z/p^f\Z)^r| 
\]
for $r \geq 2$, and we are done.
\end{proof}

\begin{theorem}\label{thm:not-cl}
Let $a \in \Z_{>0}$ and let $h \mid a$. For each natural number $n$, let $C_n$ be a deterministic $n\times n$ alternating matrix with entries in $R \coloneqq \Z/a\Z$ all divisible by $h$. Let $X_n$ be a random $C_n$-symmetric matrix with entries on and above the diagonal independent and $\varepsilon$-balanced. Let $p \mid a/h$ be prime. If the rank of $C_n/h$ mod $p$ is bounded, then we have \[
\liminf_{n\to\infty} \E[\#\Sur(\coker(X_n), (\Z/hp\Z)^2)] > |\wedge^2 (\Z/hp\Z)^2[h]|.
\]
\end{theorem}
In particular, when $h = 1$, this says that a sequence of $\varepsilon$-balanced $C_n$-symmetric matrices exhibits Cohen-Lenstra cokernel universality only if the rank of $C_n$ mod $p$ grows without bound for every $p$. Over $\Z_p$, we get the converse of Theorem~\ref{thm:intro-result} in view of Remark~\ref{rmk:gen-and-rank}.
\begin{proof}
Without loss of generality, we may assume $a = hp$, or else replace $R$ by $R/hpR$. Let $F \in \Hom(R^n, R^2)$ be uniformly random. The argument preceding Remark~\ref{rmk:gen-and-rank} shows that \[
\PP[F\text{ is isotropic for }C_n] > p^{-1} = p^{-\binom{2}{2}} = \frac{|\wedge^2 (\Z/hp\Z)^2[h]|}{|\wedge^2 (\Z/hp\Z)^2|}.
\]
In fact, $\PP[F\text{ is isotropic for }C_n]$ is bounded away from $p^{-1}$ uniformly in terms of the bound on the rank of $C_n/h$ mod $p$.

The conclusion follows from Lemma~\ref{lem:moment-from-isotropy}.
\end{proof}

\section{Computing Isotropy Probabilities}\label{sect:iso-prob}

In this subsection, we compute the probability that a uniformly random map $F\colon R^n \to G$ is a code of distance $\delta n$ and isotropic for $C$, assuming $C$ is sufficiently general. 

Proposition~\ref{prop:isotropy-condition} reduces the problem of checking that a map $F \in \Sur(R^n, G)$ is isotropic for a bilinear form $C$ to a sequence of computations in free $R_k$-modules: we just need to check that the bilinear form $F_k \circ C_k \circ F_k^*$, restricted to $R_k^{n-k+1} \subseteq G_k^*$, vanishes.

Let $h$ be a greatest common divisor of the entries of $C$. 

The idea is that $F_k\circ C_k\circ F_k^*|_{R_k^{n-k + 1}}$ is close to being uniformly distributed among alternating forms on $R_k^{n-k + 1}$ with all entries divisible by $h$. Note that such a form can be viewed as an alternating form on $(R_k/R_k[h])^{n-k+1}$. Conditional on $F_{k-1}\circ C_{k-1}\circ F_{k-1}^*|_{R_{k-1}^{n-k + 2}}$ vanishing, we get a condition over $d_{k-1}\Z/d_k\Z $ which we should expect to be satisfied about a fraction \begin{align*}
\frac{|\wedge^2((R_{k-1}/R_{k-1}[h])^{n-k+1})|}{|\wedge^2((R_k/R_k[h])^{n-k+1})|} &= \left(\frac{d_{k-1}\gcd(d_k, h)}{d_k\gcd(d_k, h)}\right)^{\binom{n - k + 1}{2}}
\end{align*}
of the time. Then the probability that all of these conditions are satisfied is about \begin{align*}
\prod_{k=1}^n \left(\frac{d_{k-1}\gcd(d_k, h)}{d_k\gcd(d_{k-1}, h)}\right)^{\binom{n - k + 1}{2}} &= \prod_{k=1}^n \left(\frac{\gcd(d_k, h)}{d_k}\right)^{\binom{n - k + 1}{2} - \binom{n - k}{2}} \\
&= \prod_{k=1}^n \left(\frac{\gcd(d_k, h)}{d_k}\right)^{\binom{n - k + 1}{2} - \binom{n - k}{2}} \\
&= \prod_{k=1}^n \left(\frac{\gcd(d_k, h)}{d_k}\right)^{n - k} = \frac{|\wedge^2G[h]|}{|\wedge^2G|}
\end{align*}
of the time.

Now $F_k^*|_{R_k^{n-k+1}}$ is just a $(n-k+1)\times n$ matrix (call it $M_k$) with entries in $R_k$. Similarly, the composition of $F_k$ with the projection of $G_k$ onto the last $n - k + 1$ coordinates is the transpose of this matrix. So $F_k \circ C_k \circ F_k^*|_{R_k^{n-k+1}} = M_k^\intercal C_k M_k$ is just a product of three matrices over $R_k$. Similarly, $F_{k-1}\circ C_{k-1}\circ F_{k-1}^*|_{R_{k-1}^{n-k + 1}}$ is obtained by reducing the entries in these three matrices modulo $d_{k-1}$. Thus, we are really interested in the probability that $M_k^\intercal C_k M_k$ vanishes when $M_k$ is uniformly random, conditioned on the entries of $M_k^\intercal C_k M_k$ being divisible by $d_k$.

\begin{proposition}\label{prop:random-alternating-form}
Let $R$ be a finite elementary divisor ring. Let $C$ be an alternating $n\times n$ matrix with entries in $R$, such that the cokernel of $C$ is generated by $g$ elements over $R$. Let $M$ be a uniformly random $n\times m$ matrix with entries in $R$. Let $S$ be an alternating $m \times m$ matrix. Then \[
\left|\PP[M^\intercal C M = S] - |R|^{-\binom{m}{2}}\right| \leq 2^{m + g + 1 - n}
\]
\end{proposition}
This proposition will be applied to $C_k/h$ with $R = R_k/R_k[h]$.
\begin{proof}
Let $C = ADB$, where $A, B \in \operatorname{GL}_n(R)$ and $D$ is diagonal. Since the cokernel of $C$ can be generated by $g$ elements, at least $n - g$ of the diagonal entries of $D$ are units in $R$.

Let $v_1, \dots, v_m$ be the columns of $M$. For $1 \leq \ell \leq m$, let $M_\ell$ be the matrix given by the first $\ell$ columns of $M$. Let $S_\ell$ be the $\ell\times 1$ matrix given by the first $\ell$ entries of the $(\ell + 1)$th column of $S$. Since $C$ is alternating, $M^\intercal C M = S$ if and only if $M_\ell^\intercal C v_{\ell + 1} = S_\ell$ for $1 \leq \ell \leq m - 1$. Since $M$ is uniformly random, so is $v_\ell$ and therefore so is $w_\ell \coloneqq Bv_\ell$ for each $\ell$. The idea is to condition on $M_\ell^\intercal$ and compute the probability that $M_\ell^\intercal ADw_{\ell+1} = S_\ell$. Note that $Dw_\ell$ is a vector whose entries are independent, and each entry is uniformly distributed inside some ideal of $R$.

Let $E_\ell$ be the event that $M_k^\intercal ADw_{k + 1} = S_k$ for $1 \leq k \leq \ell$, so $E_{m-1}$ is the event that $M^\intercal C M = S$ and $\PP[E_0] = 1$.

We have \begin{align*}
    \PP[E_\ell] &=\sum_{T}\PP[M_\ell^\intercal ADw_{\ell + 1} = S_\ell \mid M_{\ell}^\intercal = TA^{-1}]\PP[M_{\ell}^\intercal = TA^{-1}] \\
    &= |R|^{-n\ell}\sum_{T}\PP[M_\ell^\intercal ADw_{\ell + 1} = S_\ell \mid M_{\ell}^\intercal = TA^{-1}] \\
    &= |R|^{-n\ell}\sum_{T}\PP[T Dw_{\ell + 1} = S_\ell]
\end{align*}
where $(TA^{-1})^\intercal$ ranges over $n\times \ell$ matrices $\begin{bmatrix}
    t_1 & \cdots & t_\ell
\end{bmatrix}$ such that if $T_k = \begin{bmatrix}
    t_1 & \dots & t_k
\end{bmatrix}$, we have $T_k^\intercal C t_{k+1} = S_k$ for $1 \leq k < \ell$. Note that the number of such matrices $T$ is $|R|^{n\ell}\PP[E_{\ell - 1}]$.

\details{We are implicitly using that multiplication by $A^{-1}$ and transposition are bijections.}

For nice enough $T$, we will be able to show that \[
\PP[TDw_{\ell + 1} = S_\ell] = |R|^{-\ell}.
\]
Then we will bound the number of $T$ which are not nice enough to get an estimate on $\PP[E_\ell]$.

We will say $T$ is nice enough if the first $n - g$ columns of $T$ span $R^\ell$ over $R$. The following lemma shows that if $T$ is nice enough, $\PP[TDw_{\ell + 1} = S_\ell] = |R|^{-\ell}$:

\begin{lemma}\label{lem:uniform-on-cosets}
    Suppose $V$ is a finite $R$-module and $U_1, \dots, U_k$ are submodules such that $U_1 + \dots + U_k = V$. For $1 \leq i \leq k$, let $X_i$ be a uniformly random element of $U_i$, such that the $X_i$ are independent. Then $X_1 + \dots + X_k$ is a uniformly random element of $V$.
\end{lemma}
Let $t_1, \dots, t_n$ be the columns of $T$. Let $\alpha_1, \dots, \alpha_n$ be the diagonal entries of $D$. Let $X_1, \dots, X_n$ be the entries of $w_{\ell + 1}$. We have \[
TDw_{\ell+1} = \sum_{i=1}^n t_i\alpha_iX_i = \sum_{i=1}^{n-g}t_i\alpha_iX_i + \sum_{i=n-g+1}^n t_i\alpha_iX_i.
\]
Now for $1 \leq i \leq n - g$, we have $\alpha_i \in R^\times$, and since $X_i$ is uniformly distributed in $R$, we have that $t_i\alpha_iX_i$ is uniformly distributed in $Rt_i$. If $t_1, \dots t_{n-g}$ span $R^\ell$, then by Lemma~\ref{lem:uniform-on-cosets} the sum $\sum_{i=1}^{n-g}t_i\alpha_iX_i$ is uniformly distributed in $R^\ell$, so $TDw_{\ell + 1}$ is as well.

\begin{proof}[Proof of Lemma~\ref{lem:uniform-on-cosets}.]
For each $1 \leq i \leq k$, the distribution of $X_i$ in $U_i$ is translation-invariant. Then for any $v = u_1 + \dots + u_k \in V$ with $u_i \in U_i$ for each $i$, the distribution of \[
X_1 + \dots + X_k + v = (X_1 + u_1) + \dots + (X_k + u_k)
\]
is the same as the distribution of $X_1 + \dots + X_k$. Then we are done because the uniform distribution is the unique translation-invariant distribution on $V$.
\end{proof}

Now we bound the number of $T$ that are not nice enough.

\begin{lemma}\label{lem:prob-of-generating}
Suppose $V$ is an $R$-module generated by $\ell$ elements and $v_1, \dots, v_k$ are independent uniformly random elements of $V$. We have \[
\PP[Rv_1 + \dots + Rv_k \neq V] \leq 2^{\ell - k}
\]
\end{lemma}
Then the number of $(n-g)$-tuples of elements of $R^\ell$ that do not span $R^\ell$ is at most $|R|^{\ell(n-g)}\cdot2^{\ell - n + g}$, and the number of $n$-tuples such that the first $n - g$ do not span $R^\ell$ is at most $|R|^{\ell n}\cdot 2^{\ell - n + g}$. The number of $T$ satisfying the conditions above which are not nice enough is bounded above by this quantity.
\begin{proof}[Proof of Lemma~\ref{lem:prob-of-generating}]
It is possible to get a more precise bound, but this is unnecessary for our purposes. To prove this statement, we may assume $\ell < k$, or else the claim is vacuous.

Let $\mathfrak m$ be a maximal ideal of $R$ and $F = R/\mathfrak m$. Let $\overline{V} = V/\mathfrak m V$ and let $\overline{v}_i$ be the reductions of the $v_i$ modulo $\mathfrak m$, which are independent uniformly random elements of $\overline{V}$. We have $Rv_1 + \dots + Rv_k = V$ only if $F\overline{v}_1 + \dots + F\overline{v}_k = \overline{V}$. Let $d \leq \ell$ be the dimension of $\overline{V}$ over $F$.

The number of $k$-tuples of vectors that span $\overline{V}$ is the same as the number of surjections $F^k \to F^d$, which by duality is the same as the number of injections $F^d \to F^k$, which is $\prod_{i=0}^{d-1}(|F|^k - |F|^i)$. Thus, the probability that an $k$-tuple of vectors spans $\overline{V}$ is \[
\prod_{i=0}^{d-1}(1 - |F|^{i - k}) \geq \prod_{i=0}^{\ell-1}(1 - 2^{i - k}) \geq 1 - \sum_{i=0}^{\ell - 1}2^{i-k} = 1 - 2^{-k}(2^\ell - 1).
\]
\end{proof}
Now we have \begin{align*}
    \left|\PP[E_\ell] - |R|^{-\ell}\PP[E_{\ell - 1}]\right| &= \left| |R|^{-n\ell}\sum_{T}\PP[T Dw_{\ell + 1} = S_\ell] - |R|^{-\ell}\PP[E_{\ell - 1}]\right| \\
    &= \left| |R|^{-n\ell}\sum_{T\text{ nice}}\PP[T Dw_{\ell + 1} = S_\ell] + |R|^{-n\ell}\sum_{T\text{ not nice}}\PP[T Dw_{\ell + 1} = S_\ell] - |R|^{-\ell}\PP[E_{\ell - 1}]\right| \\
    &\leq \left| |R|^{-n\ell}\sum_{T\text{ nice}}|R|^{-\ell} - |R|^{-\ell}\PP[E_{\ell - 1}]\right| + 2^{\ell - n + g} \\
    &= \left||R|^{-n\ell}\sum_{T\text{ not nice}}|R|^{-\ell}\right| + \ell\cdot 2^{\ell - n + g} \\
    &\leq (1 + |R|^{-\ell})2^{\ell - n + g} \\
    &\leq 2^{\ell - n + g + 1}
\end{align*}
Iterating this process for $1 \leq \ell \leq m - 1$ gives \begin{align*}
\left|\PP[E_{m-1}] - |R|^{-\binom{m}{2}}\right| &= \left|\PP[E_{m-1}] - |R|^{-(m-1)}\PP[E_{m - 2}]\right| +  \left||R|^{-(m-1)}\PP[E_\ell] - |R|^{-(m-1)-(m-2)}\PP[E_{m - 3}]\right| + &\\
&\qquad \dots + \left||R|^{-(m-1) - \dots - 2} \PP[E_1] - |R|^{-\binom{m}{2}}\PP[E_0]\right| \\
&\leq \sum_{\ell = 1}^{m-1} |R|^{-\sum_{i=\ell + 1}^{m-1} i} 2^{\ell - n + g + 1} \\
&=\sum_{\ell = 1}^{m-1} |R|^{-\frac{(\ell + m)(m - \ell - 1)}{2}}2^{\ell - n + g + 1}.
\end{align*}
Now we have \begin{align*}
    2^{-m + n - g}\left|\PP[E_{m-1}] - |R|^{-\binom{m}{2}}\right| &\leq \sum_{\ell = 1}^{m-1} |R|^{-\frac{(\ell + m)(m - \ell - 1)}{2}}\cdot 2^{\ell - (m - 1)} 
\end{align*}
Replacing $\ell$ by $m - 1 - \ell$ in the sum, we get \begin{align*}
    2^{-m + n - g}\left|\PP[E_{m-1}] - |R|^{-\binom{m}{2}}\right| &\leq \sum_{\ell = 0}^{m - 2}|R|^{-\frac{\ell(2m - \ell + 1)}{2}}2^{-\ell} \leq \sum_{\ell = 0}^{m-2}2^{-\ell} \leq 2
\end{align*}
so that \[
\left|\PP[E_{m-1}] - |R|^{-\binom{m}{2}}\right| \leq 2^{m + g + 1 - n}
\]
as we wanted.
\end{proof}

Using Proposition~\ref{prop:random-alternating-form}, we can compute the probability that $F$ is isotropic for $C$.

\begin{corollary}\label{cor:isotropy-probability}
Use notation from Section~\ref{sect:reduction}. Let $F\colon R^n \to G$ be a uniformly random map. Let $C$ be an alternating $n\times n$ matrix and let $h$ be a greatest common divisor of its entries. Suppose that the cokernel of $C/h$ is generated by $g$ elements over $R/R[h]$. Then there are constants $N, K$ depending only on $G$ such that, as long as $n - g \geq N$ we have \[
\left|\PP[F\text{ is isotropic for }C] - \frac{|\wedge^2G[h]|}{|\wedge^2 G|}\right| \leq K\cdot 2^{g - n}
\]
\end{corollary}
\begin{proof}
Let $E'_k$ be the event that the alternating bilinear form on $G_k^*$ given by $F_k\circ C_k\circ F_k^*$ vanishes on the last $n - k + 1$ coordinates of $G_k^*$. Let $E_k = \bigcap_{1 \leq \ell \leq k} E'_\ell$ be the event that $E'_\ell$ occurs for all $1 \leq \ell \leq k$. Then $E_n$ is the event that $F$ satisfies condition (3) of Proposition~\ref{prop:isotropy-condition}. By Remark~\ref{rmk:hom-vs-sur}, we see that proving the claim with $\PP[F\text{ is isotropic for }C]$ replaced by $\PP[E_n]$ implies the original claim.

Suppose that $G$ has $k_0$ nontrivial direct summands, so that for $1 \leq k \leq n - k_0$, we have that $d_k$ is a unit in $R$. In those cases, $R_k$ is trivial, and $\PP[E_k] = 1$. Thus, we focus on the case $k > n - k_0$.

We want to compute $\PP[E_k \mid E_{k-1}] = \PP[E'_k \mid E_{k-1}]$.

Now $F_{k-1}$ is the image of $F_k$ under the quotient \[
\Hom_{R_k}(V_k, G_k) \twoheadrightarrow \Hom_{R_k}(V_k, G_k) \otimes_{R_k} R_{k-1} \cong \Hom_{R_{k-1}}(V_{k-1}, G_{k-1}),
\]
which is also the map $\Hom_{R_k}(V_k, G_k) \twoheadrightarrow \Hom_{R_k}(V_k, G_{k-1}) \cong \Hom_{R_{k-1}}(V_{k-1}, G_{k-1})$ obtained by postcomposing with the projection $G_k \twoheadrightarrow G_{k-1}$. Conditioning on $F_{k-1}$ means fixing the residue class of $F_k$ modulo $d_{k-1}$. The conditional distribution of $F_k$ given $F_{k-1}$ is uniform among maps with that residue class.

Let $\pi_k^{(m)}$ be the projection $G_k \twoheadrightarrow \Z/\min\{d_m, d_k\}\Z$ from $G_k$ onto its $m$th coordinate. Let $\pi_k^{(\geq m)} = (\pi_k^{(m)}, \pi_k^{(m+1)}, \dots, \pi_k^{(n)})$ be the projection from $G_k$ onto its last $n - m + 1$ coordinates. We observe that if $F_k$ is uniformly random, then the functions $\pi_k^{(m)} \circ F_k$ are independent uniformly random functions $V_k \to \Z/\min\{d_m, d_k\}\Z$ that together determine $F_k$. The event $E_k'$ is the vanishing of $(\pi_k^{(\geq k)} \circ F_k)\circ C_k \circ (\pi_k^{(\geq k)} \circ F_k)^*$, which is independent of $\pi_k^{(m)}\circ F_k$ for $m < k$. Thus, the conditional distribution of $\pi_k^{(\geq k)}\circ F_k$ given $F_{k-1}$ is the same as the conditional distribution of $\pi_k^{(\geq k)}\circ F_k$ given $\pi_{k-1}^{(\geq k)}\circ F_{k-1}$.

The map $\pi_k^{(\geq k)}\circ F_k$ is a uniformly random map $R_k^n \to R_k^{n-k+1}$, and $\pi_{k-1}^{(\geq k)}\circ F_{k-1}\colon R_{k-1}^n \to R_{k-1}^{n-k+1}$ is its reduction modulo $d_{k-1}$.

Let $R_k' = R_k/R_k[h]$. The form $F_k\circ C_k\circ F_k^*$ restricted to the last $n - k + 1$ coordinates of $G_k^*$ is equivalently a form on $(R_k')^{n - k + 1}$, and $E_k'$ is the event that this form vanishes.

The unconditional probability of $E_k'$ is given by Proposition~\ref{prop:random-alternating-form}, applied to $C/h$ over $R_k'$:\[
\left|\PP[E_k'] - |R_k'|^{-\binom{n - k + 1}{2}}\right| \leq 2^{g - k + 2} < 2^{g + k_0 + 2 -n}
\]
Similarly, let $E_k'' \supseteq E_k'$ be the event that $(\pi_{k-1}^{(\geq k)} \circ F_{k-1})\circ C_{k-1} \circ (\pi_{k-1}^{(\geq k)} \circ F_{k-1})^* = 0$. We have \[
\left|\PP[E_k''] - |R_{k-1}'|^{-\binom{n - k + 1}{2}}\right| < 2^{g + k_0 + 2 - n}
\]
Moreover, \[
\PP[E_k' \mid E_k''] = \PP[E_k' \mid E_{k-1}]
\]
because the random variables $\PP[E_k' \mid \pi_{k - 1}^{(\geq k)} \circ F_{k-1}]$ and $\PP[E_k' \mid F_{k-1}]$ have the same distribution. Thus \begin{align*}
    \left|\PP[E_k' \mid E_{k-1}] - \left(\frac{|R_{k-1}'|}{|R_k'|}\right)^{\binom{n - k + 1}{2}}\right| = \left|\frac{\PP[E_k']}{\PP[E_k'']} - \left(\frac{|R_{k-1}'|}{|R_k'|}\right)^{\binom{n - k + 1}{2}}\right|
\end{align*}

To bound this error, we use the following simple result. The proof is straightforward, but can be found in, e.g., \cite[Lemma 4.12]{gorokhovsky2024timeinhomogeneousrandomwalksfinite}.

\begin{lemma}\label{lem:err-combining}
    Let $x_1, \dots, x_m$ be real numbers such that $\sum_{i=1}^m |x_i| \leq \log 2$. Then \[
    \left|\prod_{i=1}^m (1 + x_i) - 1\right| \leq 2\sum_{i=1}^m |x_i|
    \]
\end{lemma}

Apply Lemma~\ref{lem:err-combining} with $x_1 = \PP[E_k']|R_k'|^{\binom{n - k + 1}{2}} - 1$ and $x_2 = \PP[E_k'']^{-1}|R_{k-1}'|^{-\binom{n - k + 1}{2}} - 1$. We have \[
|x_1| < |R_k'|^{\binom{n - k + 1}{2}}\cdot 2^{g + k_0 + 2 - n} \leq |R_k'|^{\binom{k_0}{2}}\cdot 2^{g + k_0 + 2 - n}
\]
and \[
|x_2| < \PP[E_k'']^{-1}\cdot 2^{g + k_0 + 2 - n}.
\]
We have \[
\PP[E_k''] \geq |R_{k-1}'|^{-\binom{n - k + 1}{2}} - 2^{g + k_0 + 2 - n} \geq |R_{k-1}'|^{-\binom{k_0}{2}} - 2^{g + k_0 + 2 - n}.
\]
Let $N$ be such that when $n - g \geq N$, we have $|R_{k-1}'|^{-\binom{k_0}{2}} - 2^{g + k_0 + 2 - n} \geq \frac{1}{2}|R_{k-1}'|^{-\binom{k_0}{2}}$ so that \[
|x_2| <  |R_{k-1}'|^{\binom{k_0}{2}}\cdot 2^{g + k_0 + 3 - n}
\]
Now we may potentially increase $N$ so that when $n - g \geq N$ we have $|x_1| + |x_2| \leq \log 2$. Then we get \begin{align*}
&\left|\PP[E_k' | E_{k-1}]\left(\frac{|R_k'|}{|R_{k-1}'|}\right)^{\binom{n - k + 1}{2}} - 1\right| 
&< 2\left(|R_{k-1}'|^{\binom{k_0}{2}}\cdot ^{g + k_0 + 3 - n} + |R_k'|^{\binom{k_0}{2}}\cdot 2^{g + k_0 + 2 - n}\right) \\
&\leq 3d_k^{\binom{k_0}{2}}2^{g + k_0 + 3 - n} 
\end{align*}
We now apply Lemma~\ref{lem:err-combining} again with $x_k = \PP[E_k' | E_{k-1}]\left(\frac{|R_k'|}{|R_{k-1}'|}\right)^{\binom{n - k + 1}{2}} - 1$. We have \[
\sum_{k=n - k_0 + 1}^n 3d_k^{\binom{k_0}{2}}2^{g + k_0 + 3 - n} \leq 3k_0d_n^{\binom{k_0}{2}}2^{g + k_0 + 3 - n}
\]
and we can further increase $N$ so that the right hand side is bounded above by $\log 2$. Then when $n - g \geq N$, Lemma~\ref{lem:err-combining} says that \[
\left|\prod_{k=n-k_0+1}^n\PP[E_k' | E_{k-1}]\left(\frac{|R_k'|}{|R_{k-1}'|}\right)^{\binom{n - k + 1}{2}} - 1 \right| \leq 3k_0d_n^{\binom{k_0}{2}}2^{g + k_0 + 3 - n}
\]
Now we have \[
\prod_{k=n-k_0+1}^n\PP[E_k' | E_{k-1}] = \prod_{k=n-k_0+1}^n\PP[E_k | E_{k-1}] = \PP[E_n]
\]
and \[
\prod_{k=n-k_0+1}^n\left(\frac{|R_k'|}{|R_{k-1}'|}\right)^{\binom{n - k + 1}{2}} = \prod_{k=1}^n \left(\frac{d_k\gcd(d_{k-1}, h)}{d_{k-1}\gcd(d_k, h)}\right)^{\binom{n - k + 1}{2}} =  \frac{\prod_{k=1}^n d_k^{n-k}}{\prod_{k=1}^n \gcd(d_k, h)^{n-k}} = \frac{|\wedge^2G|}{|\wedge^2G[h]|}
\]
so that \[
\left|\PP[E_n] - \frac{|\wedge^2G[h]|}{|\wedge^2G|}\right| \leq 3k_0d_n^{\binom{k_0}{2}}\frac{|\wedge^2G[h]|}{|\wedge^2G|} 2^{g + k_0 + 3 - n}.
\]
\end{proof}

Combining Corollary~\ref{cor:isotropy-probability} and Lemma~\ref{lem:moment-from-isotropy} gives: 

\begin{theorem}\label{thm:moment-convergence}
Let $a \in \Z_{>0}$ and $G$ a finite abelian group of exponent dividing $a$. Let $\varepsilon > 0$. Then there are $K, c > 0$ such that the following holds: 

Let $C$ be an alternating $n\times n$ matrix over $\Z/a\Z$ and let $h$ be a greatest common divisor of its entries. Suppose that the cokernel of $C/h$ is generated by $g$ elements over $\Z/\gcd(a, h)\Z$. Let $X$ be a random $\varepsilon$-balanced $C$-symmetric matrix. Then we have \[
\left|\E[\#\Sur(\coker(X), G)] - |\wedge^2G[h]|\right| \leq K(e^{-cn} + 2^{g - n})
\]
\end{theorem}

Finally, together with Theorem~\ref{thm:moment-robustness} and Lemma~\ref{lem:known-moments}, in the special case where $h = 1$ we get: 

\begin{corollary}\label{cor:weak-convergence}
Let $\varepsilon > 0$. For each $n$, let $C_n$ be a deterministic $n\times n$ alternating matrix with entries in $R = \Z_p$ (respectively, $R = \Z$ or $R = \widehat{\Z}$), such that the cokernel of $C_n$ is generated by $g_n$ elements. Let $X_n$ be a random $\varepsilon$-balanced $C_n$-symmetric matrix with entries in $R$. Assume \[
\lim_{n\to\infty} n - g_n = \infty
\]
Then the distribution of $\coker(X_n)$ converges weakly to the distribution of $\Gamma_{CL}^{(p, 0)}$ (respectively, $\Gamma_{CL}^{(0)}$), in the sense of Remark~\ref{rmk:weak-convergence}.
\end{corollary}

\section{Examples}\label{sect:examples}

The $C$-symmetric condition on its own is fairly restrictive. However, we emphasize that some of the power of Theorem~\ref{thm:moment-convergence} is that its dependence on the precise form of $C$ is very weak. This universality result allows us to determine cokernel distributions for various other random matrix ensembles by randomizing $C$, as the following three examples show.

\begin{example}[Universality for matrices symmetric modulo $f$]\label{ex:symm-mod-h}
The following result has been known to experts for some time, but to our knowledge has yet to be written down explicitly.

Let $R = \Z/a\Z$. Let $\varepsilon > 0$ and $h \mid a$. For each $n$, let $\widetilde{X_n}$ be a random matrix over $R$ obtained by starting with a random $\varepsilon$-balanced symmetric matrix and adding to the $ij$th entry a random variable $hy_{ij}$ for each $i < j$, where the $y_{ij}$ are independent, $\varepsilon$-balanced $R$-valued random variables. The result could be called an ``$\varepsilon$-balanced matrix symmetric modulo $h$''.

It follows from Theorem~\ref{thm:moment-convergence} and some further work that for any finite abelian group $G$ with exponent dividing $a$ we have \[
\lim_{n\to\infty}\E[\#\Sur(\coker(X_n), G)] = |\wedge^2 G[h]|.
\]
We sketch out how this can be verified.

We can describe $X_n$ as a $\varepsilon$-balanced $C_n$-symmetric matrix, where $C_n$ is a random alternating matrix whose $ij$th entry for $i < j$ is $hy_{ij}$. Suppose the cokernel of $C_n/h$ is generated by $g_n$ elements. For a finite abelian group $G$, we have \[
\E[\#\Sur(\coker(X_n), G)] = \left(\frac{a}{h}\right)^{-\binom{n}{2}}\sum_{\substack{C\text{ alternating}\\ h \mid C}}\E\left[\#\Sur(\coker(X_n), G)\ \middle|\ C_n = C\right]
\]
Nguyen and Wood \cite{nguyenLocalGlobalUniversality2025} studied the cokernel distribution of random alternating integer matrices. In particular, they show that the cokernels of random $\varepsilon$-balanced alternating matrices over $R/R[h] \cong h\Z/a\Z$ have limiting distributions depending on the parity of the matrix sizes. In particular, in the limit, the torsion part of $\coker(C_n/h)$ is almost always finite, and the rank of $C_n/h$ is almost always at least $n - 1$ \cite[Theorem 1.13]{nguyenLocalGlobalUniversality2025}. Thus, there is a constant $r$ such that $\lim_{n\to\infty}\PP[g_n > r] = 0$.   

Then we have \begin{align*}
    \left|\E[\#\Sur(\coker(X_n), G)] - |\wedge^2 G[h]|\right|
    &\leq \left(\frac{a}{h}\right)^{-\binom{n}{2}}\sum_{\substack{C\text{ alternating}\\ h \mid C}}\left|\E\left[\#\Sur(\coker(X_n), G)\ \middle|\ C_n = C\right]  - |\wedge^2G[h]|\right|\\
    &\leq \left(\frac{a}{h}\right)^{-\binom{n}{2}}\sum_{\substack{C\text{ alternating}\\ h \mid C \\ g_n(C/h) \leq r}}\left|\E\left[\#\Sur(\coker(X_n), G)\ \middle|\ C_n = C\right]  - |\wedge^2G[h]|\right| \\
    &\quad + \left(\frac{a}{h}\right)^{-\binom{n}{2}}\sum_{\substack{C\text{ alternating}\\ h \mid C \\ g_n(C/h) > r}}\left|\E\left[\#\Sur(\coker(X_n), G)\ \middle|\ C_n = C\right]  - |\wedge^2G[h]|\right|
\end{align*}
where we denote by $g_n(C/h)$ the size of a minimal generating set of $\coker(C/h)$.

To account for the second part of the sum on the right hand side, we notice that if $Y_n$ is any sequence of $\varepsilon$-balanced $C_n$-symmetric matrices, we have $\limsup_{n\to\infty}\E[\#\Sur(\coker(Y_n), G)] \leq |\wedge^2G|$ by Lemma~\ref{lem:reduction}. Then \begin{align*}
\limsup_{n\to\infty} \left(\frac{a}{h}\right)^{-\binom{n}{2}}\sum_{\substack{C\text{ alternating}\\ h \mid C \\ g_n(C/h) > r}}&\left|\E\left[\#\Sur(\coker(X_n), G)\ \middle|\ C_n = C\right]  - |\wedge^2G[h]|\right| \\&\leq \lim_{n\to\infty} \PP[g_n^0 > r](|\wedge^2G| - |\wedge^2G[h]|) = 0.
\end{align*}
In the first sum, we are working exclusively with matrices $C$ such that the cokernel of $C/h$ is generated by at most $n + r - 1$ elements. For these matrices, Theorem~\ref{thm:moment-convergence} gives an upper bound on the summands that depends only on $a, G, \varepsilon, r$. Thus, the first sum converges to 0 as well, and we get \[
\lim_{n\to\infty} \E[\#\Sur(\coker(X_n), G)] = 1
\]
as we wanted.
\end{example}

\begin{example}[Perturbation at a ``general'' small set of entries]
Let $R = \Z_p$, $R = \Z$, or $R = \Z/a\Z$. Let $\varepsilon > 0$. For each $n$, let $Y_n$ be a random $\varepsilon$-balanced symmetric matrix.

Let $k_n$ be such that $\lim_{n\to\infty} k_n = \infty$. For each $n$, fix $k_n$ entries of $Y_n$ that do not share rows or columns, and let $X_n$ be the random matrix obtained from $Y_n$ by adding some fixed units to each of these $k_n$ entries.

Let $C_n = X_n - X_n^\intercal$. This is a deterministic alternating matrix which has units in $2k_n$ entries and zeroes elsewhere, and $X_n$ is $C_n$-symmetric. The cokernel of $C_n$ is generated by $g_n = n - 2k_n$ elements, and $n - g_n =  2k_n\to \infty$. Thus, the moments of $\coker(X_n)$ agree in the limit with those of the cokernels of $\varepsilon$-balanced $n\times n$ matrices over $R$.
\end{example}

\begin{example}[Symmetric matrices with non-symmetric corners]\label{ex:random-C}
Let $R = \Z/a\Z$. Let $\widetilde{C}_n$ be drawn from the uniform distribution on the group of strictly upper-triangular $n\times n$ matrices over $R$ with entries outside the upper-left $k_n \times k_n$ corner equal to zero. Let $C_n = \widetilde{C}_n + \widetilde{C}_n^\intercal$.

Let $\widetilde{X}_n$ be drawn from the uniform distribution on symmetric $n \times n$ matrices over $R$, and let $X_n = \widetilde{X}_n + \widetilde{C}_n$.

Then $X_n$ is a matrix uniformly random among matrices which are symmetric ``outside the upper-left $k_n \times k_n$ corner'', i.e., among matrices with symmetry conditions imposed only on entries outside of this corner. Moreover, $X_n$ is $C_n$-symmetric. We will show that, as long as 
\[
\lim_{n\to\infty} k_n = \infty,
\]
we have \[
\lim_{n\to\infty}\E[\#\Sur(\coker(X_n), G)] = 1
\]
for all finite abelian groups $G$ of exponent dividing $a$.

As a consequence, if instead $X_n$ is drawn from the additive Haar measure on matrices over $\Z_p$ which are symmetric outside the upper-left $k_n \times k_n$ corner, this calculation will show that $\coker(X_n)$ converges weakly to $\Gamma_{CL}^{(p, 0)}$ provided $k_n$ grows fast enough.

For a finite abelian group $G$ of exponent dividing $a$, we have \[
\E[\#\Sur(\coker(X_n), G)] = a^{-\binom{n}{2}}\sum_{\widetilde{C}\text{ strictly upper triangular}}\E\left[\#\Sur(\coker(X_n), G)\ \middle|\ \widetilde{C}_n = \widetilde{C}\right]
\]
The matrix $C_n$ consists mostly of zeroes, with the upper-left $k_n\times k_n$ corner a uniformly random alternating $k_n\times k_n$ matrix $C_n^0$ over $R$. Then if the cokernel of $C_n^0$ is generated by $g_n^0$ elements over $R$, the cokernel of $C_n$ is generated by $g_n = n - k_n + g_n^0$ elements, and \[
n - g_n = k_n - g_n^0 
\]
An argument that is nearly identical to that in Example~\ref{ex:symm-mod-h} demonstrates that, since a random alternating matrix has a cokernel of small rank with very high probability, the primary contribution to $\E[\#\Sur(\coker(X_n), G)]$ comes from the cases where $\E\left[\#\Sur(\coker(X_n), G)\ \middle|\ \widetilde{C}_n = \widetilde{C}\right]$ is controlled by Theorem~\ref{thm:moment-convergence}.
\end{example}

A similar approach can be used to determine the limiting distributions for cokernels of $C$-symmetric matrices where $C$ is allowed to come from much more general models of random alternating matrices.

\subsection*{Acknowledgements}
The author was supported by an NSF Graduate Research Fellowship. The author thanks Mengzhen (August) Liu, Carlo Pagano, and Melanie Matchett Wood for valuable conversations. The author also thanks Yifeng Huang and Roger van Peski for helpful comments on an earlier version of this paper.

\printbibliography
\end{document}